\documentclass[11pt]{amsart}
\usepackage[T1]{fontenc}
\usepackage[utf8]{inputenc}
\usepackage{lmodern}
\usepackage{amssymb,mathtools,booktabs}
\usepackage{microtype}
\usepackage[hidelinks,unicode]{hyperref}
\hypersetup{pdftitle={Finite-rank rigidity and Nash-curve discreteness for correlation potentials},pdfsubject={AMS author-review research manuscript},pdfauthor={Josiah Park}}
\theoremstyle{plain}
\newtheorem{theorem}{Theorem}[section]
\newtheorem{lemma}[theorem]{Lemma}
\newtheorem{proposition}[theorem]{Proposition}
\newtheorem{corollary}[theorem]{Corollary}
\theoremstyle{definition}
\newtheorem{definition}[theorem]{Definition}
\newtheorem{example}[theorem]{Example}
\theoremstyle{remark}

\numberwithin{equation}{section}
\newcommand{\R}{\mathbb R}
\newcommand{\C}{\mathbb C}
\newcommand{\Q}{\mathbb Q}
\newcommand{\Prob}{\mathcal P}
\newcommand{\supp}{\operatorname{supp}}
\newcommand{\rank}{\operatorname{rank}}
\newcommand{\ord}{\operatorname{ord}}
\newcommand{\Span}{\operatorname{span}}
\allowdisplaybreaks[1]
\author{Josiah Park}
\email{josiahmicahpark@gmail.com}
\title[Finite rank and Nash-curve discreteness]{Finite-rank rigidity and Nash-curve discreteness for correlation potentials}
\date{October 28, 2026}
\begin{document}
\begin{abstract}
For every nonconstant real polynomial $P$ of degree $k$ and every positive rational exponent $a/b$, minimizing the spherical energy with potential $|P(t)|^{a/b}$ admits a discrete solution with at most
\[
\binom{d+2ak-1}{2ak}+\binom{d+2ak-2}{2ak-1}
\]
support points, where $a/b$ is reduced. This finite-replacement consequence needs no denominator threshold. Under a multiplicity-sensitive denominator condition, we prove the stronger conclusion that \emph{every} minimizer has an explicit uniform support bound; suitable negative positive-semidefinite perturbations preserve that bound. We present these consequences before the machinery. The main arguments simplify irrational-power finite-rank rigidity and give a Nash-curve ramification obstruction, stable under entire scalar additions, for powers, logarithms, and rational poles. The latter yields uniform finite supports on compact semialgebraic curves without assuming that the supports are semialgebraic. The known $p$-frame finite-replacement principle is distinguished from the extensions developed here. No unconditional discreteness theorem for general irrational $p$ is claimed.
\end{abstract}
\maketitle

\section{Introduction: discrete minimizers and their support sizes}\label{sec:intro}

Let $d\ge2$, let $Y$ be a nonempty compact subset of $S^{d-1}$, and let $f:[-1,1]\to\R$ be continuous. We minimize
\begin{equation}\label{eq:energy}
I_f(\mu)=\iint f(x\cdot y)\,d\mu(x)d\mu(y),\qquad
e_Y(f)=\min_{\mu\in\Prob(Y)}I_f(\mu).
\end{equation}
The minimum exists. A measure $\sum_iw_i\delta_{x_i}$ describes a weighted configuration; its energy is $\sum_{i,j}w_iw_jf(x_i\cdot x_j)$, including self-interactions. We allow arbitrary probability measures and arbitrary signs of $f$. All interactions are between two points. For $f(t)=|t|^p$, this is the $p$-frame problem. If $f$ is even, one may identify antipodal points, but support counts in this paper refer to points of the actual sphere unless stated otherwise.

\subsection{First consequence: a discrete optimizer with a known atom budget}

For rational polynomial-root potentials, one can replace optimization over all probability measures by optimization over a bounded number of points and their weights. The following consequence is the quickest application of the finite-rank method. It does \emph{not} require a large denominator, a semialgebraic constraint set, or any regularity of the minimizing support.

\begin{corollary}[Existence and size of a discrete minimizer]\label{cor:intro-existence}\label{prop:finite-replacement}
Let $P\in\R[t]$ be nonconstant of degree $k$, and let $a,b$ be positive coprime integers. On every nonempty compact $Y\subset S^{d-1}$, the energy with potential $f(t)=|P(t)|^{a/b}$ has a global minimizer
\begin{equation}\label{eq:intro-discrete-minimizer}
\mu_* = \sum_{i=1}^{n}w_i\delta_{x_i},\qquad
x_i\in Y,\quad w_i>0,\quad \sum_iw_i=1,\quad
n\le L_d(2ak),
\end{equation}
where
\begin{equation}\label{eq:replacement-bound}
L_d(q)=\binom{d+q-1}{q}+\binom{d+q-2}{q-1}\qquad(q\ge1).
\end{equation}
Every minimizing measure has such a minimizing replacement contained in its support. If $P(1)=0$, one atom suffices.

For the $p$-frame potential $|t|^p$, with reduced $p=a/b>0$, the atom budget improves to
\begin{equation}\label{eq:intro-pframe-budget}
R_p=
\begin{cases}
\displaystyle\binom{d+a-1}{a},&a\text{ even},\\[6pt]
\displaystyle\binom{d+2a-1}{2a},&a\text{ odd}.
\end{cases}
\end{equation}
These are existence and replacement bounds, not bounds on the support of every original minimizer.
\end{corollary}

\begin{proof}[The short argument]
Take a minimizer $\mu$, put $S=\supp\mu$, and write $K(x,y)=|P(x\cdot y)|^{a/b}$. The energy is nonnegative, so the support-positivity lemma below makes $K$ positive semidefinite on $S$. Its entrywise power satisfies
\begin{equation}\label{eq:intro-rank-identity}
K^{\circ 2b}(x,y)=P(x\cdot y)^{2a}.
\end{equation}
The sections on the right are polynomial functions of degree at most $2ak$ on the sphere, a space of dimension $L_d(2ak)$. Indeed, multiplying monomials by powers of $|x|^2=1$ reduces them to homogeneous degree $2ak$ or $2ak-1$; these restriction spaces are independent by parity. The Schur-power rank lemma below therefore gives $r:=\rank K\le L_d(2ak)$.

If $P(1)\ne0$, positivity of $K$ near the diagonal gives $E=I_f(\mu)>0$. Choose a continuous feature map $h:S\to\R^r$ and set $m=\int h\,d\mu$. The contact condition is
\begin{equation}\label{eq:feature-hyperplane}
h(x)\cdot m=E=|m|^2\quad(x\in S).
\end{equation}
Thus $h(S)$ lies in an affine hyperplane of dimension at most $r-1$. Carath\'eodory's theorem expresses $m$ using at most $r$ feature vectors from $h(S)$. The corresponding atomic measure has the same energy $|m|^2$ and is supported in $S$. If $P(1)=0$, any point mass has the zero minimum.

For $P(t)=t$, the right side of~\eqref{eq:intro-rank-identity} is homogeneous of degree $2a$, whose section rank is at most $\dim\operatorname{Sym}^{2a}(\R^d)=\binom{d+2a-1}{2a}$. If $a$ is even, already $K^{\circ b}(x,y)=(x\cdot y)^a$, giving the smaller bound. Lemmas~\ref{lem:support} and~\ref{lem:schur-rank} give the two elementary inputs used above.
\end{proof}

\paragraph{What this means in an optimization problem.}
With $R=L_d(2ak)$, the corollary gives the exact equality
\begin{equation}\label{eq:intro-finite-optimization}
e_Y(|P|^{a/b})=
\min_{\substack{x_1,\ldots,x_R\in Y\\
w_i\ge0,\ \sum_iw_i=1}}
\sum_{i,j=1}^{R}w_iw_j|P(x_i\cdot x_j)|^{a/b}.
\end{equation}
Zero weights pad a smaller configuration to $R$ terms. This is an equality of global minima, not an approximation or an algorithm for finding the minimizing configuration. For example, $|t|^{5/2}$ on $S^2$ has a minimizing measure with at most $\binom{12}{10}=66$ atoms, and $(1+t^2)^{a/b}$ on $S^{d-1}$ has one with at most $L_d(4a)$ atoms for \emph{every} positive rational $a/b$. The budget is independent of $b$ when the reduced numerator $a$ is fixed. The $p$-frame finite-replacement principle is a prior ingredient~\cite{P26}, not a new discreteness claim for all minimizers.

\subsection{Stronger consequence: a bound for every minimizer}

An additional denominator condition upgrades existence of one discrete optimizer to discreteness of \emph{every} optimizer. This is where the Nash-curve argument enters.

\begin{corollary}[Uniform support bound on the sphere]\label{cor:intro-uniform}
Keep the preceding notation and assume $P(1)\ne0$. Suppose $P$ has a complex root $\zeta$ of multiplicity $e_\zeta$ such that
\begin{equation}\label{eq:intro-denominator-condition}
\frac{b}{\gcd(b,e_\zeta)}\ge
\begin{cases}
d,&\zeta=0,\\
d+1,&\zeta\ne0.
\end{cases}
\end{equation}
Then every minimizing measure on every nonempty compact $Y\subset S^{d-1}$ satisfies
\begin{equation}\label{eq:intro-every-support-bound}
|\supp\mu|\le
\bigl(2\max\{2b,2ak\}-1\bigr)^{\,d+L_d(2ak)}.
\end{equation}
In particular, the conclusion applies to rational $p$-frame energies with $p=a/b$ and $b\ge d$.
\end{corollary}

\begin{proof}
Theorem~\ref{thm:uniform-global} bounds every positive-semidefinite restriction of the kernel by~\eqref{eq:intro-every-support-bound}. Lemma~\ref{lem:support} makes the kernel positive semidefinite on a minimizing support, since the energy is nonnegative. Compactness ensures existence of a minimizing measure.
\end{proof}

The all-minimizer bound~\eqref{eq:intro-every-support-bound} is intentionally coarse and much larger than the existence bound~\eqref{eq:intro-discrete-minimizer}. The two statements answer different questions. For example, uniform surface probability is a diffuse minimizer of the $2$-frame energy, although a discrete minimizing replacement exists.

\paragraph{An explicit family beyond $p$-frame energies.}
If $a,b$ are positive coprime integers with $b\ge d+1$, $j\ge1$, and $0\le\varepsilon\le1$, then every minimizer for
\begin{equation}\label{eq:intro-global-example}
f(t)=(1+t^2)^{a/b}-\varepsilon t^{2j}
\end{equation}
on the sphere, or on any compact sphere constraint, has at most
\begin{equation}\label{eq:intro-perturbed-bound}
N_{d,a,b}=\bigl(2\max\{2b,4a\}-1\bigr)^{\,d+L_d(4a)}
\end{equation}
support points. The bound is independent of $j$ and $\varepsilon$. More generally, the same statement holds after replacing $\varepsilon(x\cdot y)^{2j}$ by any continuous positive-semidefinite kernel $Q(x,y)$ with $\sup_xQ(x,x)\le1$; see Corollary~\ref{cor:explicit-family}. The base potential is at least one, so the modified energy is nonnegative. This is why the positive-support bound applies to the perturbed problem.

\subsection{Three routes and the role of each hypothesis}

The central question is whether a minimizing measure must have finite support. Three logically different mechanisms are relevant.
\begin{enumerate}
\item \emph{Finite-rank rigidity:} if the support kernel has finite rank, a nonintegral power may force the support itself to be finite.
\item \emph{Nash-curve obstruction:} if an infinite support contains, or accumulates along, an algebraic analytic curve, a complex branch or pole can contradict support positivity.
\item \emph{Automatic finite rank:} rational polynomial-root identities turn support positivity into a finite-rank bound, allowing Nash curve selection in an auxiliary algebraic set.
\end{enumerate}
The first mechanism does not establish its finite-rank hypothesis. In particular, the irrational-power theorem in Section~\ref{sec:irrational-rank} does not by itself yield the existence or support-size conclusions above: those use the \emph{rational} Schur-power identity~\eqref{eq:intro-rank-identity}. The second mechanism does not put a Nash arc inside an arbitrary infinite compact set. The third is an unconditional route for a substantial rational class, not for arbitrary irrational exponents.

\subsection{What is simplified and what is generalized}

The finite-rank approach to $p$-frame discreteness is developed in~\cite{P26}. We replace the specialized irrational-power rigidity input by two short lemmas: an irrational power of a nonconstant rational function cannot be rational, and a logarithmic derivative extends a finite-rank kernel identity to an algebraic closure. This proves a larger local-germ statement. If, near $t=1$,
\begin{equation}\label{eq:intro-germ}
f(t)=c\left(\frac{P(t)}{P(1)}\right)^\alpha+Q(t),
\end{equation}
where $P$ is nonconstant, $P(1)\ne0$, $Q$ is polynomial, $c\ne0$, and $\alpha$ is irrational, then finite rank of $f(x\cdot y)$ on a set forces that set to be finite (Theorem~\ref{thm:irrational-rank}). Positivity is not required. Only the germ near the diagonal matters.

For Nash arcs we use a different and more flexible principle. A Nash arc is a real-analytic parametrized curve whose coordinate functions are algebraic over the parameter. Its complexification lies on an algebraic curve, so its coordinates become meromorphic on a compact Riemann surface. Inner-product sections then have moving zeros of controlled multiplicity. A scalar singularity that cannot be removed by those multiplicities is incompatible with a holomorphic Hilbert-space representation of a positive kernel. The resulting criterion, Theorem~\ref{thm:nash-obstruction}, is unchanged by adding \emph{any entire scalar function real on the real axis}; no smallness or positivity condition is imposed on that addition.

This gives concrete non-$p$-frame families. For $a>0$, each of
\begin{equation}\label{eq:intro-examples}
\log(a+t^2)+H(t),\qquad
\frac1{a+t^2}+H(t),\qquad
(a+t^2)^\alpha+H(t)\quad(\alpha\notin\Q)
\end{equation}
has finite minimizing supports on every compact semialgebraic curve in the sphere. Here $H$ is entire and real on $\R$, and $\alpha$ is real. In fact the support sizes have a common finite bound for each fixed potential and admissible curve (Theorem~\ref{thm:curve-minimizers}). The support need not itself be semialgebraic.

The full-sphere conclusions have already been stated in Corollaries~\ref{cor:intro-existence} and~\ref{cor:intro-uniform}, with their respective atom budgets. Theorem~\ref{thm:rational-roots} proves the effective-denominator criterion, and Theorem~\ref{thm:uniform-global} makes the all-minimizer bound explicit. Neither requires any geometric assumption on the support. The perturbation principle in Proposition~\ref{prop:psd-stability} explains the non-$p$-frame family~\eqref{eq:intro-global-example} and its bound~\eqref{eq:intro-perturbed-bound}.

\subsection{Prior ingredients and scope of the contribution}

Support positivity and the harmonic approach to the original discreteness problem are established tools~\cite{BGM+21,BGM+22}. The real $p$-frame finite-replacement, denominator, and irrational finite-rank results in~\cite{P26} are starting points, not claimed anew here. Algebraic-power rigidity is related to the Umehara-algebra framework in~\cite{CDY17}; our finite-rank proof does not invoke that theorem. Nash curve selection is standard real algebraic geometry~\cite[Proposition 8.1.13]{BCR98}; normalization and meromorphic divisors are standard Riemann-surface theory~\cite{F81}. We prove explicitly the local Hilbert-norm criterion and continuation step, in the spirit of the squared-norm method underlying Calabi's criterion~\cite{Calabi1953}.

The emphasized extensions are the polynomial-power finite-rank theorem, the ramification formulation and its entire-perturbation stability, the uniform support bound on semialgebraic curves, and the bounded-degree uniform estimate and perturbation consequences for polynomial roots. The root-denominator and finite-replacement statements also streamline variants already present in the supplied working notes; they are supporting results rather than stand-alone novelty claims. The rational-root section can be read after Sections~\ref{sec:preliminaries} and~\ref{sec:moving-zero}, without the Hilbert-space continuation argument. The accompanying contribution note records earlier-draft overlap and the limits of the source check.

\section{Two elementary support and rank facts}\label{sec:preliminaries}

For a continuous symmetric kernel $K$ on a compact space, write $I_K(\mu)=\iint K\,d\mu\,d\mu$. Positive semidefiniteness means that every finite matrix $[K(x_i,x_j)]$ is positive semidefinite. Kernel rank is the supremum of the ranks of these finite matrices; symmetry is assumed, but positivity need not be.

\begin{lemma}[The centered support kernel]\label{lem:support}
If $\mu$ minimizes $I_K$, with energy $E$ and support $S$, then its potential $U_\mu(x)=\int K(x,y)\,d\mu(y)$ satisfies $U_\mu\ge E$ on the admissible set and $U_\mu=E$ on $S$. Moreover,
\begin{equation}\label{eq:support-positivity}
\sum_{i,j}c_ic_jK(x_i,x_j)\ge E\left(\sum_i c_i\right)^2
\qquad(x_i\in S,\ c_i\in\R).
\end{equation}
Thus $K-E$ is positive semidefinite on $S$. If $E\ge0$, then $K$ is positive semidefinite there as well.
\end{lemma}

\begin{proof}
Variation toward a point mass gives $U_\mu\ge E$; integrating this inequality and using continuity gives equality on the support. Approximate each $\delta_{x_i}$ by the normalized restriction $\rho_{i,n}$ of $\mu$ to a shrinking positive-mass neighborhood. Put $c=\sum_i c_i$ and $\eta_n=\sum_i c_i\rho_{i,n}-c\mu$. This is centered and has bounded density relative to $\mu$, so both signs of $\mu+t\eta_n$ are feasible for small $|t|$. The linear energy term vanishes, and the quadratic term is nonnegative. Expanding that term and passing to the limit proves~\eqref{eq:support-positivity}.
\end{proof}

The centering is indispensable for signed potentials. None of the curve arguments below assumes that the original kernel is positive semidefinite on a minimizing support when its minimum is negative.

\begin{lemma}[A Schur-power rank bound]\label{lem:schur-rank}
If $K$ is positive semidefinite and $m\ge1$ is an integer, then
\begin{equation}\label{eq:schur-rank}
\rank K\le\rank(K^{\circ m}),
\end{equation}
where $K^{\circ m}(x,y)=K(x,y)^m$.
\end{lemma}

\begin{proof}
For a finite Gram representation $K(x_i,x_j)=v_i\cdot v_j$, the power is the Gram matrix of $v_i^{\otimes m}$. If $v_1,\ldots,v_r$ are independent, take dual linear functionals $\ell_j(v_i)=\delta_{ij}$. Applying $\ell_j^{\otimes m}$ to a relation among the tensor powers proves their independence. Take suprema over finite subsets.
\end{proof}

Positivity cannot be omitted: $(1+st)^{1/2}$ has infinite rank on a small real interval, whereas its square has rank two. Its infinitely many nonzero Taylor coefficients show the infinite rank; their alternating signs explain the failure of positivity.

We will also use the following finite-rank representation. If a real symmetric kernel has rank $r<\infty$, choose $r$ anchors $z_i$ with invertible principal Gram matrix $G$. Then
\begin{equation}\label{eq:anchor-representation}
K(x,y)=k(x)^TG^{-1}k(y),\qquad k_i(x)=K(x,z_i).
\end{equation}
Indeed, the Schur complement after adjoining two further points is zero. Such anchors exist because a finite symmetric rank-$r$ matrix has an invertible principal minor of size $r$. For a continuous positive kernel,~\eqref{eq:anchor-representation} gives the continuous Euclidean feature map $G^{-1/2}k(x)$. For a general symmetric kernel it gives a finite-dimensional representation with a possibly indefinite bilinear form.

\section{An elementary irrational-power rank theorem}\label{sec:irrational-rank}

The proof uses algebraic geometry only through one familiar fact: a nonconstant meromorphic function on a compact Riemann surface has a zero or pole of nonzero integer order. In particular, the argument does not need positivity or a global single-valued fractional power.

\begin{lemma}[Rational functions cannot have irrational power ratios]\label{lem:rational-power}
Let $V$ be an irreducible complex algebraic variety, and let $R,S$ be nonzero rational functions on $V$. If, on a nonempty analytic open subset,
\begin{equation}\label{eq:rational-power}
S=R^\alpha
\end{equation}
for a holomorphic branch and $\alpha\notin\Q$, then $R$ and $S$ are constant.
\end{lemma}

\begin{proof}
If $R$ is nonconstant, choose a smooth point in the open subset where $dR\ne0$. A generic algebraic curve section through that point has nonconstant $R$. On the normalization of its projective closure, the restrictions are meromorphic and the local identity implies the global differential identity
\begin{equation}\label{eq:log-differential}
\frac{dS}{S}=\alpha\frac{dR}{R}.
\end{equation}
At a zero or pole of $R$ of nonzero integer order $m$, residues give $n=\alpha m$ for an integer $n$, a contradiction. Thus $R$ is constant, and so is $S$.
\end{proof}

\begin{lemma}[Finite-separable power rigidity]\label{lem:power-rigidity}
Let $A(z,w)$ and $B(z,w)$ be holomorphic near $(0,0)\in\C^2$ and admit finite separable expansions, that is, finite sums of a function of $z$ times a function of $w$. Suppose $A(0,0)=B(0,0)=1$. If $B=A^\alpha$ for the branch equal to one at the origin and $\alpha\notin\Q$, then
\begin{equation}\label{eq:rank-one}
A(z,w)=A(z,0)A(0,w).
\end{equation}
If additionally $A$ is symmetric and $A(t,t)=1$ for real $t$ near zero, then $A\equiv1$ near the origin.
\end{lemma}

\begin{proof}
Write $A(z,w)=u(z)\cdot v(w)$ and $B(z,w)=g(z)\cdot h(w)$ with finite holomorphic vectors; no positivity or bilinear normal form is needed. Let $V$ be the complex Zariski closure of $z\mapsto(u(z),g(z))$. It is irreducible, since evaluation of polynomials on a connected holomorphic disk has a prime kernel. For its coordinate vectors $(U,G)$, set
\[
a_w(U,G)=U\cdot v(w),\qquad b_w(U,G)=G\cdot h(w).
\]
Differentiating $B=A^\alpha$ in $w$ and clearing denominators gives
\begin{equation}\label{eq:polynomial-log-identity}
(\partial_wb_w)a_w-\alpha b_w(\partial_wa_w)=0.
\end{equation}
For each fixed $w$, this polynomial identity holds on the original image and therefore on $V$. Near the image of $z=0$, integrate its logarithmic derivative in $w$ to obtain
\begin{equation}\label{eq:ratio-identity}
\frac{b_w}{b_0}=\left(\frac{a_w}{a_0}\right)^\alpha.
\end{equation}
The factors are nonzero on a sufficiently small relative neighborhood and for small $w$; the branch equals one when $w=0$. For each fixed small $w$, both ratios are rational functions on $V$. Lemma~\ref{lem:rational-power} makes them constant. Evaluation on the original image and at $z=0$ gives~\eqref{eq:rank-one}.

By symmetry, $A(t,t)=A(t,0)^2=1$. Continuity and $A(0,0)=1$ give $A(t,0)=1$ near zero. The holomorphic identity theorem then gives $A\equiv1$.
\end{proof}

\begin{theorem}[Irrational polynomial-power rigidity]\label{thm:irrational-rank}
Let $f$ be continuous on $[-1,1]$. Suppose its germ at $1$ has the form~\eqref{eq:intro-germ}, with $P,Q\in\R[t]$, $P$ nonconstant, $P(1)\ne0$, $c\ne0$, and $\alpha\in\R\setminus\Q$. The quotient $P(t)/P(1)$ is positive near $1$, and its positive real power is intended. For every $S\subset S^{d-1}$,
\begin{equation}\label{eq:rank-finiteness}
\rank\bigl[f(x\cdot y)\bigr]_{x,y\in S}<\infty
\quad\Longleftrightarrow\quad |S|<\infty.
\end{equation}
No positive-semidefiniteness assumption is required.
\end{theorem}

\begin{proof}
Only the forward implication requires proof. Finite rank is preserved on the compact closure by continuity of all finite minors. If $S$ is infinite, work in a small cap about an accumulation point. Subtract the finite-rank polynomial kernel $Q(x\cdot y)$ and divide by $c$. The resulting finite-rank kernel is locally
\begin{equation}\label{eq:local-power-kernel}
\left(\frac{P(x\cdot y)}{P(1)}\right)^\alpha.
\end{equation}
It is real analytic in both ambient variables on a small neighborhood.

The analytic enlargement lemma in Appendix~\ref{app:analytic-enlargement} produces a nonconstant real-analytic sphere arc $x(t)$ on which the same finite-rank identity holds. Complexify it and set
\begin{equation}\label{eq:polynomial-base}
A(z,w)=\frac{P(x(z)\cdot x(w))}{P(1)},\qquad B(z,w)=A(z,w)^\alpha.
\end{equation}
The polynomial kernel $A$ has finite rank, and the enlargement gives finite rank for $B$. Both have diagonal one. Lemma~\ref{lem:power-rigidity} gives $P(x(s)\cdot x(t))=P(1)$. Since $P$ is nonconstant, the real solutions of $P(u)=P(1)$ are isolated; continuity near the diagonal forces $x(s)\cdot x(t)=1$. Real unit vectors with inner product one coincide, contradicting the nonconstant arc.
\end{proof}

The enlargement arc in this proof is \emph{real analytic}, not necessarily Nash. This distinction is deliberate. For irrational powers the feature functions need not be algebraic. The theorem proves an obstruction to finite rank, not finite rank of an arbitrary minimizing support.

\section{A ramification obstruction on Nash arcs}\label{sec:nash}

This section removes the finite-rank assumption but imposes an algebraic curve structure. We isolate the only scalar information about the potential that the argument needs.

\begin{definition}[Removable ramification orders]\label{def:ramification}
Let $F$ be a holomorphic germ at $1$, continuable along every path in $\C\setminus\Sigma$, where $\Sigma$ is finite and $1\notin\Sigma$. For $\zeta\in\Sigma$, let $\mathcal R_\zeta(F)$ consist of the positive integers $\ell$ for which some branch obtained from this continuation has
\begin{equation}\label{eq:ramified-germ}
w\longmapsto F(\zeta+w^\ell)
\end{equation}
extending holomorphically to $w=0$. Agreement on a sufficiently small sector is enough to define the branch in this statement. A logarithm, a pole, or an irrational power at $\zeta$ has no such removable order. A positive rational power with reduced denominator $b$ has precisely the positive multiples of $b$.
\end{definition}

The definition quantifies over all continuation branches. One must not select a singular branch if another reachable branch would become holomorphic at the same point.

\subsection{A moving zero with bounded multiplicity}\label{sec:moving-zero}

\begin{lemma}[Moving-zero lemma]\label{lem:moving-zero}
Let $x_1,\ldots,x_d$ be meromorphic on a connected compact Riemann surface $C$, real on a nonconstant real arc $x(t)$, and satisfy $\sum_i x_i^2=1$. For $\zeta\ne1$, put
\begin{equation}\label{eq:pencil}
\begin{gathered}
A_t(z)=\sum_i x_i(z)x_i(t),\qquad r_\zeta=\dim V_\zeta,\\
V_\zeta=\Span_\C\{A_t-\zeta:t\text{ in the real arc}\}.
\end{gathered}
\end{equation}
Outside any prescribed finite subset of $C$, there are a real parameter $t_0$ and a point $z_0$ where $A_{t_0}-\zeta$ has a zero of order
\begin{equation}\label{eq:small-zero}
1\le\ell\le r_\zeta-1.
\end{equation}
In particular $r_0\le d$, and $r_\zeta\le d+1$ when $\zeta\ne0$.
\end{lemma}

\begin{proof}
All members of $V_\zeta$ have pole divisors bounded by one fixed effective divisor. A basis has a nonzero meromorphic Wronskian: in characteristic zero, vanishing of the Wronskian of one-variable meromorphic functions is equivalent to linear dependence over the constants. This follows by induction after division by the first basis element. Outside a finite set containing the poles and Wronskian exceptions, a nonzero member of $V_\zeta$ cannot have a zero of order $r_\zeta$ or greater, since its first $r_\zeta$ jets would vanish.

Suppose every zero of every $A_t-\zeta$ lies in that finite exceptional set, enlarged by the prescribed forbidden points. Uniform pole bounds also bound the total zero order. There are therefore only finitely many possible divisors. Two meromorphic functions on a compact connected surface with the same divisor differ by a constant. Hence the continuous family of projective classes $[A_t-\zeta]$ is constant.

This is impossible. Write $C(s,t)=A_t(s)-\zeta$. Its diagonal is $1-\zeta\ne0$. A constant projective pencil and symmetry give, for a fixed real $t_*$,
\begin{equation}\label{eq:constant-pencil}
C(s,t)=\frac{C(s,t_*)C(t_*,t)}{1-\zeta}.
\end{equation}
The diagonal identity implies $C(t,t_*)^2=(1-\zeta)^2$. On a connected small interval $C(t,t_*)=1-\zeta$, so $A_t(s)=1$ on the real arc. It would be constant as a real unit-vector arc, a contradiction. Thus a zero occurs outside the exceptional set, and the Wronskian gives~\eqref{eq:small-zero}.
\end{proof}

The compactification is essential: zeros confined to a finite set do not force finitely many projective classes without a common pole bound on a compact surface.

\subsection{The Hilbert-norm continuation step}

\begin{lemma}[Local positivity and continuation]\label{lem:hilbert-continuation}
The following facts hold.
\begin{enumerate}
\item If a real-analytic symmetric kernel $K(s,t)$ is positive semidefinite on an infinite real parameter set accumulating at zero, all finite blocks of its Taylor coefficient matrix at $(0,0)$ are positive semidefinite. Its polarization consequently has a local holomorphic Hilbert-space feature map.
\item Let $\Phi$ be a real-valued real-analytic function on a connected Riemann surface $\Omega$. If locally on one nonempty disk $\Phi=\|F\|^2$ for a holomorphic Hilbert-valued map, then such representations exist locally at every point. The feature map can be continued along any fixed path after constant unitary identifications on overlaps.
\end{enumerate}
No global single-valued feature map is asserted.
\end{lemma}

\begin{proof}
For the first assertion, choose $n+1$ distinct nodes in the given set approaching zero. Apply the invertible Newton divided-difference transform to rows and columns of their Gram matrix. Its $(i,j)$ entry tends to
\begin{equation}\label{eq:coefficient-matrix}
c_{ij}=\frac{\partial_s^i\partial_t^jK(0,0)}{i!j!}.
\end{equation}
Congruence and passage to the limit preserve positivity. Real analyticity gives a Cauchy bound $|c_{ij}|\le M\rho^{-i-j}$ on a smaller bidisk. Choose Hilbert vectors $v_i$ with $v_i\cdot v_j=c_{ij}$. The series $F(z)=\sum_i v_i z^i$ converges in Hilbert norm for $|z|<\rho$ and supplies the feature map.

For the second assertion, expand $\Phi(z,\bar z)$ about each point. The same Gram-series argument says that a local squared-norm representation is equivalent to positivity of every finite coefficient block. This property is closed, because the coefficients vary continuously with the center in a fixed coordinate chart. It is open, because an existing representation remains a representation nearby. Connectedness therefore propagates it to all of $\Omega$.

On a connected overlap, equality of squared norms gives equality of the polarized kernels. The correspondence between feature vectors extends to an isometry of their closed spans. Replace each target by the closed span of its image; the span over any open subdisk is already the whole target, by the scalar holomorphic identity theorem. The overlap isometry is thus unitary. Successive identifications along a finite disk cover of a path continue the initial map. There may be unitary monodromy around loops, which is harmless here.
\end{proof}

\begin{theorem}[Nash-arc ramification obstruction]\label{thm:nash-obstruction}
Let $f:[-1,1]\to\R$ be continuous and agree near $1$ with a real holomorphic germ $F$. Suppose $F$ is continuable along every path in $\C\setminus\Sigma$, where $\Sigma$ is finite and disjoint from $[1,\infty)$, and use its real-analytic continuation along that interval. Suppose some $\zeta\in\Sigma$ satisfies
\begin{equation}\label{eq:bad-ramification}
\mathcal R_\zeta(F)\cap\{1,\ldots,D_\zeta\}=\varnothing,
\qquad
D_\zeta=
\begin{cases}
d-1,&\zeta=0,\\
d,&\zeta\ne0.
\end{cases}
\end{equation}
For every nonconstant Nash sphere arc $\gamma$, every real constant $c$, and every infinite parameter set accumulating at an interior parameter, the kernel
\begin{equation}\label{eq:shifted-arc-kernel}
f(\gamma(s)\cdot\gamma(t))-c
\end{equation}
is not positive semidefinite on that parameter set. For a specified arc, $D_\zeta$ may be replaced by $r_\zeta-1$ from Lemma~\ref{lem:moving-zero}.

The criterion is unchanged by adding any entire function $H$ real on $\R$ to the germ $F$. It is also unchanged by multiplying $F$ by a nonzero real constant.
\end{theorem}

\begin{proof}
Assume positivity, translate the accumulating parameter to zero, and shrink so all real pairwise correlations lie in the prescribed germ neighborhood. Lemma~\ref{lem:hilbert-continuation} supplies a local Hilbert feature map on a whole parameter disk. Complexify the Nash arc, take its irreducible algebraic curve closure, and normalize its projective compactification. If the initial parametrization is critical, move within this disk to a regular real parameter with nonsingular image. Such parameters exist because the map is nonconstant and the algebraic curve has finitely many singular points. A local inverse then transfers the feature map to a disk of the normalization. Its coordinate functions $x_i$ are meromorphic on a compact connected Riemann surface $C$, and
\begin{equation}\label{eq:complex-sphere}
\sum_i x_i(z)^2=1.
\end{equation}
Let $\Omega$ be $C$ minus their finite pole set. It is connected. At every point of $\Omega$,
\begin{equation}\label{eq:norm-range}
q(z):=\sum_i|x_i(z)|^2\ge\left|\sum_i x_i(z)^2\right|=1.
\end{equation}
Consequently
\begin{equation}\label{eq:global-real-analytic-norm}
\Phi(z)=F(q(z))-c
\end{equation}
is a globally defined real-analytic real-valued function on $\Omega$. Near the initial arc it is the squared norm of the local feature map. Lemma~\ref{lem:hilbert-continuation} continues that map along any prescribed path in $\Omega$.

By Lemma~\ref{lem:moving-zero}, choose a real $t_0$ in the initial small arc and a point $z_0\in\Omega$ where $A(z)=\sum_i x_i(z)x_i(t_0)$ satisfies $A(z_0)=\zeta$ with order $1\le\ell\le r_\zeta-1$. Every $A$ for a real base point is nonconstant: if it were constant, evaluation at that base point would make it identically one and force the real unit-vector arc to be constant. Thus the preimages of the finite set $\Sigma$ are finite. Choose a path from the base point $t_0$ to $z_0$ in $\Omega$ avoiding these preimages until its endpoint.

Continue the initial feature map $\mathcal F$ along this path, and pair it with the fixed initial vector $\mathcal F(t_0)$. The resulting scalar function is holomorphic through $z_0$ and initially equals $F(A(z))-c$. Uniqueness of analytic continuation identifies it along the path with the corresponding branch of this scalar expression. Adding $c$ therefore makes a branch of $F(A(z))$ holomorphic at $z_0$.

Locally $A(z)-\zeta=(z-z_0)^\ell a(z)$ with $a(z_0)\ne0$. Taking a holomorphic $\ell$th root of $a$ gives a local coordinate $w$ for which $A(z)=\zeta+w^\ell$. We have produced an order $\ell\in\mathcal R_\zeta(F)$, contradicting~\eqref{eq:bad-ramification}.

Finally, an entire addition is holomorphic after every such local substitution and cannot change removability. It preserves real analyticity along $[1,\infty)$. A nonzero real scalar multiplication likewise does not change removability. The same proof applies.
\end{proof}

The continuation through $z_0$ comes from the Hilbert feature map, not from presuming that the singular scalar expression extends there. This is the branch issue that the argument must resolve.

\section{Discreteness on curves and for semialgebraic supports}\label{sec:curve-consequences}

Throughout this section, $f$ satisfies Theorem~\ref{thm:nash-obstruction}. No sign condition on $f$ or its minimum is imposed.

\begin{corollary}[Semialgebraic supports]\label{cor:semialgebraic-support}
On any nonempty compact admissible set $Y\subset S^{d-1}$, every minimizing measure whose support is semialgebraic has finite support. More generally, a minimizing support cannot contain infinitely many points of one nonconstant Nash arc with parameter accumulation at an interior point.
\end{corollary}

\begin{proof}
Lemma~\ref{lem:support} makes $f(x\cdot y)-e_Y(f)$ positive semidefinite on the support. This contradicts Theorem~\ref{thm:nash-obstruction} in the stated accumulating-arc situation. An infinite semialgebraic set contains a nonconstant Nash arc by Nash curve selection or semialgebraic cell decomposition, so it cannot be such a support.
\end{proof}

A stronger unconditional conclusion is available when the admissible set itself is a curve. Its minimizing support may be a priori highly nonsemialgebraic.

\begin{lemma}[A finite confluent-Gram obstruction]\label{lem:jet-cap}
Let $K_c(s,t)=K(s,t)-c$ be real analytic near a diagonal point $(t_0,t_0)$. If for some $n$ the matrix
\begin{equation}\label{eq:jet-matrix}
C_n(t_0,c)=\left[
\frac{\partial_s^i\partial_t^jK_c(t_0,t_0)}{i!j!}
\right]_{0\le i,j\le n}
\end{equation}
has a negative quadratic direction, there is an interval about $t_0$ on which every positive-semidefinite restriction of $K_c$ contains at most $n$ distinct parameter points. The conclusion remains valid for $c'$ in a sufficiently small neighborhood of $c$.
\end{lemma}

\begin{proof}
For any $n+1$ distinct nearby nodes, transform their Gram matrix by the invertible Newton divided-difference matrix on both sides. The entries are mixed divided differences. As all nodes tend to $t_0$, they converge to~\eqref{eq:jet-matrix}, uniformly without any lower bound on their separation. The Hermite--Genocchi formula expresses an $i$th divided difference as the simplex integral of the $i$th derivative, with simplex volume $1/i!$; apply it in both variables~\cite{dB05}. Continuity of the mixed derivatives gives the asserted uniformity. A fixed negative direction therefore stays negative for all sufficiently nearby node tuples and shifts $c'$. Congruence preserves positivity, so those node tuples cannot support a positive-semidefinite Gram matrix.
\end{proof}

\begin{lemma}[Closed Nash parametrizations]\label{lem:closed-charts}
A compact semialgebraic set $Y$ of dimension at most one in the sphere is covered by finitely many isolated points and images $\gamma_j([0,1])\subset Y$, where each $\gamma_j$ is nonconstant and Nash on an open neighborhood of $[0,1]$, with values in the sphere there.
\end{lemma}

\begin{proof}
Nash cell decomposition gives finitely many one-dimensional cells parametrized by Nash maps $\eta:(0,1)\to Y$, and finitely many points. Each coordinate of $\eta$ is a bounded algebraic function. It has a finite limit and a convergent Puiseux expansion with nonnegative exponents at each endpoint. Choose an integer $M$ divisible by all the denominators in these endpoint expansions and substitute
\begin{equation}\label{eq:endpoint-parametrization}
\chi_M(t)=\frac{t^M}{t^M+(1-t)^M},\qquad 0\le t\le1.
\end{equation}
This is an increasing bijection, analytic near $[0,1]$. At zero it is $t^M$ times a positive analytic unit; at one, $1-\chi_M(t)$ has the analogous form. Thus $\eta\circ\chi_M$ extends analytically through both endpoints and remains algebraic, hence Nash. The sphere equation extends by analyticity, and compactness puts the endpoint images in $Y$. These finitely many maps cover the closures of the cells. Every remaining point is isolated: otherwise it would belong to the closure of a one-dimensional cell.
\end{proof}

\begin{theorem}[Uniform finite supports on semialgebraic curves]\label{thm:curve-minimizers}
Let $Y\subset S^{d-1}$ be nonempty, compact, semialgebraic, and of dimension at most one. For every real $c_0$ there is a finite number $N(f,Y,c_0)$ bounding the cardinalities of all subsets $S\subset Y$ on which $f(x\cdot y)-c$ is positive semidefinite, uniformly for every $c\ge c_0$.

In particular, all minimizers of $I_f$ on $\Prob(Y)$ have at most $N(f,Y,\min f)$ support points. No semialgebraicity assumption on the support is needed.
\end{theorem}

\begin{proof}
Use the finitely many charts of Lemma~\ref{lem:closed-charts} and fix the shift $c_0$. At any parameter $t_0$, if every coefficient block~\eqref{eq:jet-matrix} of
$K_{c_0}(s,t)=f(\gamma_j(s)\cdot\gamma_j(t))-c_0$ were positive semidefinite, the Gram-series construction in Lemma~\ref{lem:hilbert-continuation} would make it a positive kernel on a neighborhood. Theorem~\ref{thm:nash-obstruction} forbids this, even at a critical parameter. Some finite block therefore has a negative direction. Lemma~\ref{lem:jet-cap} supplies a local cardinality cap. Analytic extension through the endpoints lets us treat them in exactly the same way.

Compactness of each parameter interval gives finitely many local caps. Sum them over the charts and add the isolated points. For each distinct image point choose one preimage; distinct image points have distinct chosen parameters, so self-intersections or multiple parametrizations do not invalidate the bound. This bounds every positive set for $f-c_0$.

If $c\ge c_0$ and $f-c$ is positive semidefinite on $S$, then
\begin{equation}\label{eq:shift-monotonicity}
f-c_0=(f-c)+(c-c_0)\mathbf1
\end{equation}
is positive semidefinite there as well, where $\mathbf1$ denotes the constant kernel. Thus the same cap works for every $c\ge c_0$. Finally, $e_Y(f)\ge\min f$, and Lemma~\ref{lem:support} applies.
\end{proof}

The theorem supplies a finite-jet certificate for a local bound, but not a universal closed-form expression for $N(f,Y,c_0)$. The bound can depend on the entire perturbation. A semialgebraic constraint set of dimension at least two is not covered by this argument: its minimizing support need not be semialgebraic or contain a Nash arc.

\begin{corollary}[Unrestricted positive-kernel subtraction on curves]\label{cor:curve-psd-stability}
Under the preceding hypotheses on $f$ and $Y$, let $Q$ be any continuous positive-semidefinite kernel on $Y$. Every minimizer for
\[
K(x,y)=f(x\cdot y)-Q(x,y)
\]
has finite support, with a common bound. More precisely, if $\sup_xQ(x,x)\le D$, all these supports have at most $N(f,Y,\min f-D)$ points. No smallness, zonality, or finite-rank assumption on $Q$ is needed.
\end{corollary}

\begin{proof}
Positive semidefiniteness gives $|Q(x,y)|\le\sqrt{Q(x,x)Q(y,y)}\le D$, so the modified minimum $E$ is at least $\min f-D$. On a minimizing support, $K-E$ is positive semidefinite. Adding $Q$ makes $f-E$ positive semidefinite there, and Theorem~\ref{thm:curve-minimizers} gives the bound.
\end{proof}

\subsection{Scalar examples beyond powers of an inner product}

The following singularity tests make the criterion concrete. In each case entire real perturbations and nonzero real scalar multiples are allowed.

\begin{example}[Logarithmic and rational potentials]\label{ex:log-rational}
Let $P\in\R[t]$ be nonconstant and positive on $[1,\infty)$, and assume the chosen potential is continuous on $[-1,1]$. For $F=\log P$, every complex root $\zeta$ of multiplicity $e$ contributes the local term $e\log(t-\zeta)$ plus a holomorphic germ. After ramification of order $\ell$, its derivative has nonzero residue $e\ell$. It cannot be the derivative of a holomorphic germ, so $\mathcal R_\zeta(F)=\varnothing$.

For a real rational function with a finite pole $\zeta$ and no poles on $[1,\infty)$, every ramification retains a pole. Hence again $\mathcal R_\zeta(F)=\varnothing$. In particular, both $\log(a+t^2)+H(t)$ and $(a+t^2)^{-1}+H(t)$, $a>0$, satisfy the curve and semialgebraic-support conclusions for every entire real $H$.
\end{example}

\begin{example}[Polynomial powers]\label{ex:polynomial-powers}
Suppose $P\in\R[t]$ is nonconstant and positive on $[1,\infty)$, and use the real germ $F=P^\alpha$. At a root $\zeta$ of multiplicity $e$, it has the form $(t-\zeta)^{e\alpha}a(t)$ on each branch, with $a$ holomorphic and nonzero. If $\alpha$ is irrational, no positive ramification order is removable. If $\alpha<0$, no order is removable either, since the resulting power has negative exponent. For positive rational $\alpha$, removability is exactly the condition $e\alpha\ell\in\mathbb Z_{\ge0}$.

Thus a positive rational exponent gives the obstruction when its effective denominator at a root is at least $d$ for $\zeta=0$, or at least $d+1$ for $\zeta\ne0$. For $F(t)=t^p$, this recovers the Nash-arc obstruction for every irrational $p>0$, and for reduced rational $p=a/b$ with $b\ge d$. For $a>0$, $(a+t^2)^\alpha+H(t)$ is a continuous concrete family for every real irrational $\alpha$.
\end{example}

The same test covers a single-valued essential singularity at a finite point: composition with a finite branched covering does not make such a singularity removable. The finite singularity set and the analytic continuation along $[1,\infty)$ must still be checked; real analyticity only on $[-1,1]$ is not enough.

\section{Unconditional rational-root discreteness and perturbations}\label{sec:rational-global}

For rational roots, finite rank is automatic on positive supports. One can therefore select a Nash arc in an algebraic feature graph even if the original set is not semialgebraic. A direct meromorphic-order proof also removes the diagonal-continuation assumption of Theorem~\ref{thm:nash-obstruction}.

\begin{theorem}[Multiplicity-sensitive rational-root criterion]\label{thm:rational-roots}
Let $P\in\R[t]$ be nonconstant, and let $a,b$ be positive coprime integers. Put $f(t)=|P(t)|^{a/b}$. If $f(x\cdot y)$ is positive semidefinite on an infinite $S\subset S^{d-1}$, then $P(1)\ne0$ and every complex root $\zeta$ of $P$, of multiplicity $e_\zeta$, satisfies
\begin{equation}\label{eq:root-threshold}
\frac{b}{\gcd(b,e_\zeta)}\le
\begin{cases}
d-1,&\zeta=0,\\
d,&\zeta\ne0.
\end{cases}
\end{equation}
Consequently, if $P(1)=0$, or if any inequality~\eqref{eq:root-threshold} fails, every minimizing measure on every nonempty compact $Y\subset S^{d-1}$ has finite support.
\end{theorem}

\begin{proof}
Pass to the closure of $S$. If $f(1)=0$, positivity and the zero diagonal force the whole support kernel to vanish. Since $P$ is nonzero, there is a neighborhood of $1$ containing no other root. Distinct sufficiently close sphere points have inner product in that neighborhood but below $1$, contradicting the vanishing kernel. Compactness then forces $S$ to be finite. Hence an infinite positive set has $c=f(1)>0$.

On $S$ the positive kernel $K=f(x\cdot y)$ satisfies
\begin{equation}\label{eq:root-schur-identity}
K(x,y)^{2b}=P(x\cdot y)^{2a}.
\end{equation}
The right side has finite polynomial section rank, so Lemma~\ref{lem:schur-rank} gives finite rank for $K$. Let $h:S\to\R^R$ be its continuous feature map, with $|h(x)|^2=c$. Take the real Zariski closure $Z$ of its graph in $\R^{d+R}$. It is compact because $|x|^2=1$ and $|u|^2=c$ hold on it. Moreover,
\begin{equation}\label{eq:graph-pair-identity}
(u\cdot v)^{2b}=P(x\cdot y)^{2a}
\qquad((x,u),(y,v)\in Z).
\end{equation}
To justify the two-variable extension, first fix a graph point in one variable and take Zariski closure in the other, then reverse the roles. If two points of $Z$ have the same $x$,~\eqref{eq:graph-pair-identity} and their norms imply $u=\pm v$. Thus the projection to $x$ has finite fibres.

An infinite compact real algebraic set contains a nonconstant Nash arc. Its $x$-projection is nonconstant because the fibres are finite. Normalize and compactify the complex algebraic curve of this arc. Its coordinates $x_i,u_j$ are meromorphic on a compact connected Riemann surface, and
\begin{equation}\label{eq:meromorphic-pair-identity}
\begin{gathered}
B(s,t)^{2b}=P(A(s,t))^{2a},\\
A(s,t)=\sum_i x_i(s)x_i(t),\quad
B(s,t)=\sum_j u_j(s)u_j(t).
\end{gathered}
\end{equation}
These identities follow by continuation of~\eqref{eq:graph-pair-identity}.

Fix a root $\zeta$ of multiplicity $e_\zeta$. Lemma~\ref{lem:moving-zero}, with all coordinate poles excluded, gives a real $t_0$ and a zero of $A(\cdot,t_0)-\zeta$ of order $1\le\ell\le r_\zeta-1$. Taking orders in~\eqref{eq:meromorphic-pair-identity} yields
\begin{equation}\label{eq:order-divisibility}
b\,\ord B(\cdot,t_0)=a e_\zeta\ell,
\qquad
\frac{b}{\gcd(b,e_\zeta)}\mid\ell.
\end{equation}
The denominator is therefore at most $r_\zeta-1$, which gives~\eqref{eq:root-threshold}.

For the minimizing-measure conclusion, when $f(1)>0$, every probability has positive energy: a small neighborhood of a support point has positive kernel on its square. The minimum is attained and positive, so Lemma~\ref{lem:support} makes the uncentered kernel positive on every minimizing support. If $f(1)=0$, the minimum is zero and the same positivity follows directly. Apply the positive-support result.
\end{proof}

The proof does not require $P$ to have no root on $[1,\infty)$. That assumption was needed only for the Hilbert-norm continuation route when no finite-rank features are available.

Without a denominator restriction, the finite-rank identity still gives the discrete minimizing replacement proved in Corollary~\ref{cor:intro-existence}. That conclusion needs no Nash arc: it is the rank bound followed by Carath\'eodory's theorem. What the denominator condition adds is that an infinite minimizing support is impossible. We now quantify this stronger assertion.

\subsection{An explicit uniform full-sphere bound}

The preceding argument proves finiteness but does not control the degree of its Zariski closure. A bounded-degree algebraic enlargement repairs this and yields a uniform cardinality estimate. We first record a standard coarse counting argument.

\begin{lemma}[Finite real zero sets]\label{lem:finite-real-zeros}
If real polynomials of degree at most $m\ge1$ in $n$ variables have a finite common real zero set $Z$, then
\begin{equation}\label{eq:real-zero-bound}
|Z|\le(2m-1)^n.
\end{equation}
\end{lemma}

\begin{proof}
Let $V$ be the sum of their squares, so $\deg V\le2m$ and $V^{-1}(0)=Z$. The empty case is immediate. Otherwise choose disjoint closed balls about the points of $Z$, with $V$ strictly positive on their boundaries. For all sufficiently small vectors $\lambda$, each ball contains an interior minimum of $V(z)-\lambda\cdot z$. Choose such a $\lambda$ to be a regular value of $\nabla V$, using Sard's theorem. The resulting critical points are nonsingular, hence isolated complex solutions of the $n$ polynomial equations $\partial_jV=\lambda_j$, each of degree at most $2m-1$. B\'ezout's inequality for isolated complex solutions bounds their number by $(2m-1)^n$, proving~\eqref{eq:real-zero-bound}. The full complex solution set need not be finite.
\end{proof}

\begin{theorem}[Uniform rational-root support bound]\label{thm:uniform-global}
Let $P,a,b$ satisfy Theorem~\ref{thm:rational-roots}, with $P(1)\ne0$, and suppose a root violates~\eqref{eq:root-threshold}. Put $k=\deg P$ and
\begin{equation}\label{eq:uniform-parameters}
R_0=L_d(2ak),\qquad m=\max\{2b,2ak\},\qquad
N=(2m-1)^{d+R_0}.
\end{equation}
Every subset of $S^{d-1}$ on which $|P(x\cdot y)|^{a/b}$ is positive semidefinite has at most $N$ points. In particular the same bound holds for every minimizing support on every compact sphere constraint.
\end{theorem}

\begin{proof}
The rank bound gives a feature map $h:S\to\R^{R_0}$, with zero padding if necessary and $|h(x)|^2=c:=|P(1)|^{a/b}$. For $z=(x,u)$ and $w=(y,v)$ in $\R^{d+R_0}$, the pair polynomial
\[
T(z,w)=(u\cdot v)^{2b}-P(x\cdot y)^{2a}
\]
has degree at most $m$ separately in each variable. Let $\nu(z)$ be the vector of all monomials of degree at most $m$, including the constant. Then $T(z,w)=\nu(z)^TJ\nu(w)$ for a fixed real symmetric matrix $J$.

Let $W$ be the linear span of $\nu(x,h(x))$ for $x\in S$. The polynomial pair identity on the graph says that every pair of vectors in $W$ has zero $J$-pairing. Hence the compact real algebraic set
\begin{equation}\label{eq:bounded-degree-envelope}
Z_W=\{(x,u):|x|^2=1,\ |u|^2=c,\ \nu(x,u)\in W\}
\end{equation}
contains the graph and satisfies $T=0$ on its square.

The projection of $Z_W$ to $x$ has fibres of size at most two, by the same Cauchy--Schwarz argument as before. If $Z_W$ were infinite, a Nash arc with nonconstant $x$-projection would yield the meromorphic identity~\eqref{eq:meromorphic-pair-identity}. The moving-zero argument would again force $b/\gcd(b,e_\zeta)$ to divide a positive order below its stated threshold. This is impossible. Thus $Z_W$ is finite. We have used its polynomial identity, not an assertion that the original positive-root kernel is positive semidefinite on its whole projection.

Choose linear equations defining $W$ in the monomial-coordinate space. Together with the two norm equations,~\eqref{eq:bounded-degree-envelope} is therefore defined by polynomials of degree at most $m\ge2$ in $d+R_0$ variables. Lemma~\ref{lem:finite-real-zeros} gives $|S|\le|Z_W|\le N$. A fibre need not contain both signs, so no additional factor of two is removed.
\end{proof}

This bound is a complexity estimate, not a practical atom count. It grows rapidly with numerator and denominator and supplies no uniform bound along rational approximations to an irrational exponent.

If $P(1)=0$, uniformity instead follows by separation. A positive restriction has zero diagonal and is identically zero. If $r<1$ is the largest root of $P$ in $[-1,1)$, its distinct support points have Euclidean distance at least $\delta=\sqrt{2(1-r)}$. Disjoint balls of radius $\delta/2$ lie in the ball of radius $1+\delta/2$, giving the bound $\lfloor(1+2/\delta)^d\rfloor$. If there is no such root, at most one point is possible.

\subsection{A stability principle in the order of positive kernels}

\begin{proposition}[Negative positive-semidefinite perturbations]\label{prop:psd-stability}
Let $X$ be compact and let $H$ be a continuous symmetric kernel such that every subset carrying a positive-semidefinite restriction of $H$ is finite. Let $Q$ be a continuous positive-semidefinite kernel. If
\begin{equation}\label{eq:stability-energy-condition}
\min_{\mu\in\Prob(X)} I_{H-Q}(\mu)\ge0,
\end{equation}
then every minimizer of $I_{H-Q}$ has finite support. If all positive-semidefinite restrictions of $H$ have at most $N$ points, the same $N$ bounds every modified minimizing support. A sufficient condition is
\begin{equation}\label{eq:stability-diagonal-condition}
\sup_{x\in X}Q(x,x)\le e_X(H).
\end{equation}
The kernel $Q$ need not be zonal, polynomial, or finite rank.
\end{proposition}

\begin{proof}
By Lemma~\ref{lem:support} and~\eqref{eq:stability-energy-condition}, $H-Q$ is positive semidefinite on a minimizing support. Adding $Q$ makes $H$ positive there, so the support is finite. To verify the sufficient condition, positivity gives $|Q(x,y)|\le\sqrt{Q(x,x)Q(y,y)}$, and hence $I_Q(\mu)\le\sup_xQ(x,x)$. Thus $I_{H-Q}(\mu)\ge e_X(H)-\sup_xQ(x,x)$.
\end{proof}

\begin{corollary}[A full-sphere non-$p$-frame family]\label{cor:explicit-family}
Let $a,b$ be positive coprime integers with $b\ge d+1$. Let $Q$ be any continuous positive-semidefinite kernel on $S^{d-1}$ with $\sup_xQ(x,x)\le1$. Then every minimizing measure for
\begin{equation}\label{eq:nonpframe-family}
K(x,y)=(1+(x\cdot y)^2)^{a/b}-Q(x,y)
\end{equation}
has finite support. In particular this holds for~\eqref{eq:intro-global-example}, or for a subtraction $\sum_jc_j(x\cdot y)^{2j}$ with $c_j\ge0$ and $\sum_jc_j\le1$.
All such minimizing supports have a common bound from Theorem~\ref{thm:uniform-global}, depending only on $d,a,b$, not on $Q$.
\end{corollary}

\begin{proof}
The polynomial $1+t^2$ has two simple nonzero complex roots. Theorem~\ref{thm:rational-roots} therefore excludes every infinite positive support when $b\ge d+1$. Its base potential is at least one, so $e(H)\ge1$. Proposition~\ref{prop:psd-stability} applies. The even monomials are positive kernels by their tensor representations.
\end{proof}

For example, on $S^2$ every minimizing measure for $(1+(x\cdot y)^2)^{5/4}-\varepsilon(x\cdot y)^{2k}$ is finite for $0\le\varepsilon\le1$. Exponents larger than one are included; the theorem is not restricted to concavity of the scalar function of $(x\cdot y)^2$.

For the original $p$-frame kernel with reduced $p=a/b$ and $b\ge d$, the same stability principle applies whenever $\sup_xQ(x,x)\le d^{-\max\{1,p/2\}}$. Indeed $I_2\ge1/d$; for $p\le2$, $I_p\ge I_2$, while for $p\ge2$, Jensen's inequality gives $I_p\ge I_2^{p/2}\ge d^{-p/2}$.

\section{Why the hypotheses cannot be discarded}\label{sec:limits}

\subsection{Small rational denominators allow positive Nash arcs}

Let $v(t)=(1,t)/\sqrt{1+t^2}$ on a small interval, and put $x(t)=v(t)^{\otimes b}$ in the symmetric tensor space of dimension $b+1$. This is a nonconstant Nash sphere arc. For reduced $p=a/b>0$,
\begin{equation}\label{eq:veronese-example}
|x(s)\cdot x(t)|^{a/b}
=(v(s)\cdot v(t))^a
=\langle v(s)^{\otimes a},v(t)^{\otimes a}\rangle,
\end{equation}
after choosing the interval so $v(s)\cdot v(t)>0$. Thus an infinite positive restriction exists in dimension $b+1$. The threshold at a root at zero is sharp as a general positive-support criterion; this is not a claim that the arc supports a global minimizing measure.

The extra constant in a nonzero-root pencil matters. On the Nash circle parametrization $x(t)=((1-t^2)/(1+t^2),2t/(1+t^2))$,
\begin{equation}\label{eq:shifted-root-example}
\sqrt{1+x(s)\cdot x(t)}=\sqrt2\,v(s)\cdot v(t)
\end{equation}
on a sufficiently short interval. Hence denominator two at the nonzero root $-1$ does not prohibit a positive arc in dimension two. One cannot replace the nonzero-root threshold $d+1$ by $d$ in general.

\subsection{Nonpolynomial and analytic do not mean discrete}

For $f(t)=e^t$, uniform surface probability $\sigma$ is the unique minimizer on $S^{d-1}$ and is diffuse. Its potential is constant by rotation invariance. If $\eta=\mu-\sigma$, uniform convergence of the exponential series gives
\begin{equation}\label{eq:exponential-counterexample}
I_f(\mu)-I_f(\sigma)
=\sum_{n=0}^{\infty}\frac1{n!}
\left\|\int x^{\otimes n}\,d\eta(x)\right\|^2\ge0.
\end{equation}
Equality makes all polynomial moments of $\eta$ vanish, hence $\eta=0$ by Stone--Weierstrass. The finite-singularity criterion does not apply to this entire potential.

Even a nonpolynomial analytic germ can give finite rank on a Nash arc. For a real noninteger $\lambda$, the germ $f(t)=\cos(\lambda\arccos t)$ is analytic at $1$ and continuous on $[-1,1]$. On a short circle arc with angle parameters $s,t$, its kernel is $\cos(\lambda(s-t))$, of rank two. The circle image has a Nash parametrization, although these feature functions need not be Nash. Thus neither ``nonpolynomial'' nor ``irrational parameter'' can replace the precise power or ramification hypotheses.

\subsection{Why continuation and sign assumptions matter}

The ray condition in Theorem~\ref{thm:nash-obstruction} cannot be replaced by real analyticity near the physical interval. For $a>1$,
\begin{equation}\label{eq:ray-counterexample}
\frac1{a-t}=\sum_{n=0}^{\infty}\frac{t^n}{a^{n+1}}
\end{equation}
defines a positive-semidefinite kernel on the entire sphere by its tensor-power expansion. Its pole survives every ramification, but lies on the forbidden ray $[1,\infty)$.

Likewise the nonnegative root in Theorem~\ref{thm:rational-roots} cannot be replaced by an arbitrary signed root. On $S^1$, the potential $f(t)=-(2-t)^{1/3}$ has a simple nonzero root of effective denominator $3=d+1$, yet uniform arclength is its unique minimizer. Indeed
\[
f(t)=-2^{1/3}+\sum_{n\ge1}a_nt^n,\qquad
a_n=2^{1/3-n}(-1)^{n+1}\binom{1/3}{n}>0.
\]
The series converges absolutely and uniformly on $[-1,1]$. The polynomial $\cos^n\theta$ has nonnegative Fourier coefficients and a positive coefficient at frequency $n$, so $f(\cos\theta)$ has strictly positive coefficients at every nonzero frequency. Its energy is minimized precisely when all those Fourier moments vanish, which characterizes arclength. The root at $2$ also violates the ray hypothesis, so this example contradicts neither theorem.

\subsection{The remaining \texorpdfstring{$p$}{p}-frame gap}

For an arbitrary irrational $p>0$ in higher dimension, support positivity alone does not supply finite rank, semialgebraicity of the support, or an accumulating Nash arc. An infinite compact set need not contain any arc. Theorems~\ref{thm:irrational-rank} and~\ref{thm:nash-obstruction} rule out two mechanisms for an infinite minimizer, but they do not prove that one of those mechanisms must occur. The general non-even discreteness problem remains unresolved by these arguments.

The perturbation statements must also be kept distinct. Entire scalar additions preserve the Nash-curve obstruction without a size or sign restriction. On compact semialgebraic curves one may also subtract arbitrary continuous positive-semidefinite kernels, by Corollary~\ref{cor:curve-psd-stability}. The full-sphere statement in Proposition~\ref{prop:psd-stability} instead requires a nonnegative perturbed minimum. We do not claim unconditional full-sphere discreteness under arbitrary entire additions.

\subsection*{Provenance and review status}
This AI-assisted author-review manuscript develops a focused set of deductions from the supplied finite-rank and two-point working papers. Standard support, algebraic, and analytic ingredients are identified separately from the emphasized generalizations. The accompanying contribution and source note records the internal proof-review scope and literature-check limitations. Internal checking is not external refereeing, formal verification, or a definitive priority assessment.

\appendix

\section{Analytic enlargement of a finite-rank restriction}\label{app:analytic-enlargement}

We give the precise curve-selection step used in Theorem~\ref{thm:irrational-rank}. It is a real-analytic statement and makes no assertion that the selected curve is algebraic.

\begin{lemma}[Analytic enlargement]\label{lem:analytic-enlargement}
Let $S$ be an infinite subset of the real unit sphere, with accumulation point $x_0$. Suppose a symmetric kernel is real analytic on an ambient neighborhood of $(x_0,x_0)$ and has finite rank on the nearby part of $S$. Then there is a nonconstant real-analytic sphere arc through $x_0$ on which the same local kernel has finite rank.
\end{lemma}

\begin{proof}
Restrict to a sufficiently small neighborhood and choose finite anchors with invertible Gram matrix $G$ for the restricted rank. The zero-rank case is treated with an empty feature vector. By~\eqref{eq:anchor-representation}, the analytic residual
\begin{equation}\label{eq:analytic-residual}
R(x,y)=K(x,y)-k(x)^TG^{-1}k(y)
\end{equation}
vanishes on the square of the nearby set. Choose a common ambient neighborhood $U$ where all terms are analytic, including the sections against the anchors.

Let $\mathcal I$ be the ideal of convergent real-analytic germs at $x_0$ that vanish on $S$ in some neighborhood of $x_0$, where the neighborhood may depend on the germ. The ring of convergent power series is Noetherian~\cite{KP02}. Choose finite generators $q_1,\ldots,q_m$ and representatives on a common neighborhood. Their zero set contains all sufficiently nearby points of $S$ and has a nonisolated real germ at $x_0$. Real-analytic curve selection gives a nonconstant analytic arc $\gamma$ through $x_0$ in that zero set. Because $|x|^2-1\in\mathcal I$, the arc lies on the sphere after shrinking.

For a fixed nearby $y\in S$, the germ $R(\cdot,y)$ belongs to $\mathcal I$. Hence $R(\gamma(t),y)=0$ near zero, initially on an interval possibly depending on $y$. The function is analytic on a fixed connected arc interval, so it vanishes on the whole interval. Consequently, for every fixed parameter $t$, the germ $R(\gamma(t),\cdot)$ vanishes on nearby $S$ and belongs to $\mathcal I$. Apply the same one-variable identity argument a second time to obtain
\begin{equation}\label{eq:arc-residual}
R(\gamma(s),\gamma(t))=0
\end{equation}
on the product of a fixed small interval with itself. Thus the anchor formula holds on the arc and gives finite rank. This two-step argument avoids assuming a common neighborhood for all germs in the ideal.
\end{proof}

\section*{Acknowledgment of AI assistance}
GPT-6 Astra, accessed through Tasklet, was used to assist with mathematical exploration, the development and checking of arguments, and the drafting and revision of this manuscript, including its \LaTeX{} source. Responsibility for the final mathematical statements, proofs, and references remains with the author.

\bibliographystyle{amsplain}
\bibliography{p-frame-references}
\end{document}